\documentclass{article}

\usepackage{tikz-cd}
\usepackage[english]{babel}

\usepackage{fullpage}

\usepackage{enumitem}
\usepackage{comment}
\usepackage{esvect}
\usepackage{amsmath}
\usepackage{amssymb}
\usepackage{amsfonts}
\usepackage{amsthm}
\usepackage{graphicx}
\usepackage{csquotes}
\usepackage{mathtools}
\usepackage[colorlinks=true, allcolors=blue]{hyperref}
\usepackage[capitalize,noabbrev]{cleveref}
\usepackage{xcolor}
\usepackage{soul}
\usepackage{caption}
\usepackage{float}
\usepackage{subcaption}

\theoremstyle{plain}
\newtheorem{theorem}{Theorem}
\newtheorem*{theorem*}{Theorem}
\newtheorem{lemma}[theorem]{Lemma}
\newtheorem*{lemma*}{Lemma}

\newtheorem{question}[theorem]{Question}

\theoremstyle{remark}
\newtheorem{remark}[theorem]{Remark}

\def\lb{[\![}
\def\rb{]\!]}

\newcommand{\naturals}{\mathbb{N}}
\newcommand{\integers}{\mathbb{Z}}

\newcommand{\N}{\mathbb{N}}

\newcommand{\Z}{\mathbb{Z}}

\newcommand{\triangulations}[2]{\mathcal{M}_{#1, #2}}
\newcommand{\percnp}[2]{\mathcal{M}_{#1, #2}^{\mathrm{perc}}}
\newcommand{\perc}{\mathcal{M}^{\mathrm{perc}}}
\newcommand{\krewerasnk}[2]{\mathcal{K}_{#1, #2}}
\newcommand{\kreweras}{\overset{\leftarrow}{\mathcal{K}}}
\newcommand{\krewerasbar}{\overline{\mathcal{K}}}
\newcommand{\rootedge}{\vv{e}_r}
\newcommand{\peelsteps}[1]{S\left({#1}\right)}

\title{From Kreweras walks to branching perimeter processes of percolated triangulations}
\author{Renan Gross \and Emmanuel Kammerer \and Frederik Ravn Klausen \thanks{DPMMS, University of Cambridge. \{rg751, ek672, frk23\}@cam.ac.uk}}

\begin{document}
\maketitle

\begin{abstract}
In this note, using a result of Bernardi, Holden and Sun, we give an explicit geometric relation between the perimeter of the peeling process along the percolation interface of a triangulation, and the corresponding Kreweras walk. The relation naturally extends to the branching peeling exploration. This sheds light on the relation between the growth-fragmentation process and correlated Brownian excursions discovered by Da Silva, Powell and Watson.
\end{abstract}

\begin{figure}[H]
    \centering
    \includegraphics[width=0.745\linewidth]{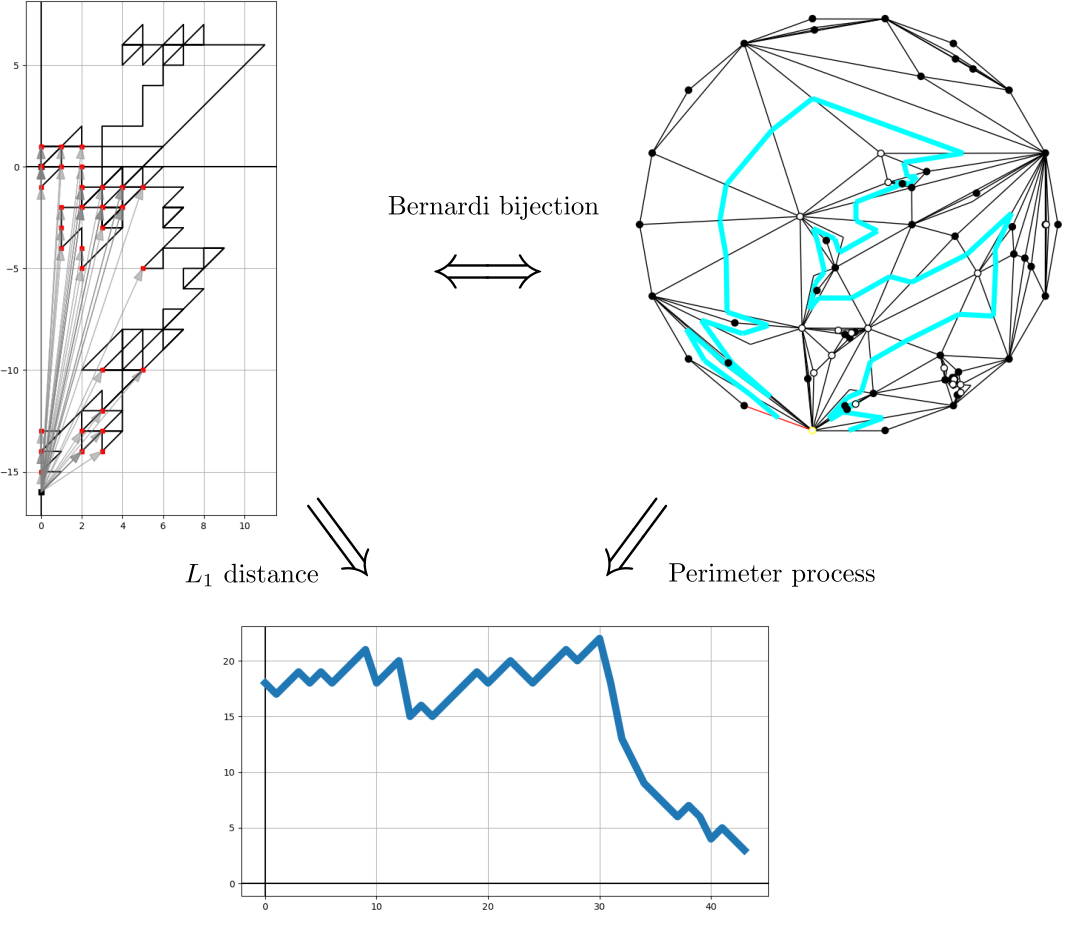}
    \caption{\textbf{Top left}: a Kreweras walk with $200$ steps. \textbf{Top right}: the equivalent percolated triangulation, along with its percolation interface. \textbf{Bottom}: the length of the perimeter when exploring along the interface. The top equivalence and bottom-right implication are known; in this note, we show that the perimeter curve is obtained from the Kreweras walk by the $L_1$ distance of certain non-excursion points (marked by red squares) to the endpoint (vectors denoted by grey arrows). }
    \label{fig:correspondence}
\end{figure}

\section{Introduction}
In \cite{DSPW25}, it was discovered that a certain type of branching process, known as a self-similar growth-fragmentation, is embedded in a two-dimensional Brownian excursion subject to some geometric conditions. This result has a natural interpretation in terms of Liouville quantum gravity (LQG) via the mating of trees \cite{DMS21}, where the growth-fragmentation records the boundary length of an exploration in LQG. The purpose of this note is to prove that the embedding actually holds in the discrete setting.

Let us already state our main result, even though some of the exact definitions will only be given in the next section. We denote the integers by $\integers$ and the positive integers by $\naturals$. For integers $a \leq b$, denote by $\lb a,b \rb$ the set $\{a, a+1, \ldots, b\}$.

A percolated planar triangulation $(M, \sigma)$ with boundary of perimeter $p$ is a planar map $M$ made of one $p$-gon and finitely many triangles, equipped with a percolation configuration $\sigma$ on the set of vertices of the map. The map $(M, \sigma)$ has a distinguished oriented edge called root edge oriented from a white vertex to a black vertex on the boundary and such that the $p$-gon lies on the right-hand side of the root edge. Consider the percolation interface starting from the root edge. The peeling exploration along the interface starts from the $p$-gon and is obtained by adding one by one the triangles along the interface. Sometimes, the triangle that is discovered separates the unexplored region into two regions. In this case, one of these two regions does not contain the rest of the interface and is added to the explored region. For integer $ n \geq 0$, we denote by $P(n)$ the number of edges on the boundary of the unexplored region after $n$ steps. Note that $P(0)=p$. The exploration stops when the last triangle of the interface is discovered.

A Kreweras walk is a (finite) walk in $\Z^2$ with steps in $\{(1,0), (0,1), (-1, -1)\}$. The bijection from \cite{Ber07, BHS23} maps every percolated planar triangulation $(M,\sigma)$ with $p-1$ black vertices and only one white vertex on the boundary to a Kreweras walk that starts from $(0,0)$, ends at $(0, -p+2)$, and stays in the quadrant $\{(x,y) \in \integers^2: x\ge 0, y \ge -p+2\}$. Moreover, each step in $\{(1,0), (0,1) \}$ corresponds to an internal triangle of $M$.

Let $X:\lb 0, N \rb \to \integers^2$ be a Kreweras walk. We say that a time $n\in \lb 0, N \rb$ is in a \emph{cone excursion} if there exists $n_1 \in \lb n+1, N \rb$ such that for all $m \in \lb n, n_1-1 \rb$, we have $X(m) \in X(n_1) + \naturals^2$. The step $X(n+1)-X(n)$ is a \emph{spine-step} if $n$ is not in a cone excursion and  $X(n+1)-X(n) \in \{(1,0), (0, 1)\}$.

\begin{theorem} \label{thm:explicit_perimeter}
    Let $(M, \sigma) $ be a percolated planar triangulation where the boundary has exactly one white vertex, and let $X: \lb 0, N \rb \to \integers^2$ be its corresponding Kreweras walk (of length $N$). Denote by $L$ be the number of spine steps of $X$, and for $n \in \lb 1,L\rb$, let $T(n)$ be the time of $n$-th spine step of $X$. Then the perimeter process along the percolation interface of $(M,\sigma)$ has $L$ steps in total and is given by
    \begin{equation*}
        P(n) = \lVert  X(T(n+1))-X(N) \rVert_1 + 2\, 
    \end{equation*}
    for all $n \in \lb 0, L-1 \rb$.
\end{theorem}
In the above theorem, the peeling process linearly follows the percolation interface. However, one can also consider a branching peeling exploration (following the so-called exploration tree), where instead of ``filling in'' when the unexplored region is separated into two, the exploration branches and explores both sides. In this setting, for every triangle $t$ of $M$, we can define the perimeter process $(P_t(n))_{n \in \lb 0, L_t-1 \rb}$ of the exploration branch towards $t$, where $L_t$ is the total number of triangles discovered by this branch. \cref{thm:explicit_perimeter} can  be extended to this setting as well.
\begin{theorem}[informal; see \cref{sec:branching} for a precise statement]\label{thm:branching}
    Let $(M, \sigma) $ be a percolated planar triangulation where the boundary has exactly one white vertex, and let $X$ be its corresponding Kreweras walk. Let $t$ be a triangle of $M$. Then
    \[
    P_t(n)=\lVert X(T_t(n+1))- X(D_t(n+1)) \rVert_1 +2 \, ,
    \]
    where $T_t(n)$ is the time of the $n$-th spine step in a restriction $X\vert_{\lb 0, m_t \rb}$ of the Kreweras walk for some $m_t$ defined using the Bernardi bijection, and $D_t(n)$ is the closure time of the excursion of $X$ that contains the trajectory $X\vert_{\lb T_t(n), m_t \rb}$.
\end{theorem}
\begin{remark}
    The proof of \cref{thm:explicit_perimeter} shows that the first coordinate of the vector $X(T(n))-X(N)$ is the number of white vertices on the boundary of the unexplored region, while the second coordinate gives the number of black vertices on the explored region. Using the same proofs, \cref{thm:explicit_perimeter} and \cref{thm:branching} extend to percolated planar triangulations with boundary conditions given by a connected set of $\ell\ge 1$ black vertices and $r\ge 1$ white vertices. We did not state the result at this level of generality in order to emphasize the connection with \cite{DSPW25}.
\end{remark}

Several other models of random planar maps equipped with a statistical mechanics model can be encoded using a bijection with a walk, e.g.,
the more general Hamburger-Cheeseburger bijection for FK-decorated triangulations from \cite{She16}. One can consider the peeling exploration along the interface there as well. 
\begin{question}
Does an analogue of \cref{thm:explicit_perimeter} or \cref{thm:branching} hold in more general discrete statistical mechanics models? 
\end{question}
In the continuum, an analogous question is left open by \cite{DSPW25}.

\subsection{Motivation from random planar maps and Liouville quantum gravity}
Percolation on random triangulations is relatively well understood, thanks to several breakthroughs from different directions.

One approach to random triangulations is to study their scaling limits. Le Gall showed that, properly scaled, a uniformly random triangulation converges to the Brownian sphere \cite{LG13}, while Albenque, Holden and Sun showed that triangulations with boundary converge to the Brownian disk \cite{AHS20}. Another milestone was the identification of the Brownian sphere with a Liouville quantum gravity (LQG) of parameter $\gamma= \sqrt{8/3}$ \cite{MS20, MS21, MS21b, MS21c}. The scaling limit of the percolation interfaces was first described in \cite{BHS23}: the outer-boundaries of the clusters are given by an independent conformal loop ensemble of parameter $\kappa=6$ and the scaling limit of the associated exploration tree is given by a branching Schramm-Loewner Evolution (SLE$_6$) process. The key used to obtain this scaling limit are a bijection between percolated triangulations and a family of walks, called Kreweras walks, first identified by \cite{Ber07}, and the mating of trees correspondence discovered in \cite{DMS21}, which encodes an SLE-decorated LQG using a planar Brownian excursion. The scaling limit was further strengthened: the convergence was proven in \cite{GHS21} to hold jointly with the scaling limit towards the Brownian disk, and the convergence towards the $\sqrt{8/3}$-LQG was then proved to hold under the Cardy embedding in \cite{HS23}.

Another approach to the study of percolation on random triangulations is the peeling exploration introduced in \cite{Wat95, Ang03}, and generalised in \cite{Bud16}, see also \cite{Cur23}. The peeling exploration can be used to study various properties of triangulations, such as the growth of metric balls starting from the boundary, the behaviour of the random walk, as well as percolation. One can indeed explore the triangulation by following the percolation interface as in \cite{AC15, CR20,BC22}. The perimeter process is a crucial quantity in the study of the peeling exploration. Its scaling limit was obtained in \cite{CG17}. In general, when deciding which face to explore next, the peeling exploration may leave unexplored ``holes'' in the triangulation which are never explored. To discover the whole triangulation, one can define a branching peeling algorithm and record the perimeter processes along all the branches. In \cite{BCK18}, the scaling limit of the perimeter processes along all these branches was obtained and can be described by a branching process called a self-similar growth-fragmentation, which also appears in \cite{LGR20} in the study of metric balls in the Brownian sphere.

Even though some connections between the two approaches were already made at the discrete level between bijections with walks and peeling explorations (see \cite{BHS23}), the two scaling limits described above (mating of trees and growth-fragmentations) have only been related very recently. The first identification of a growth-fragmentation in a mating of trees was obtained in \cite{AdS22,AHPS21} in the critical case $\gamma=2$. More recently, the same growth fragmentation as in \cite{BCK18} was shown in \cite{DSPW25} to be embedded in the Brownian excursion encoding the mating of trees in the case $\gamma=\sqrt{8/3}$. In particular, they obtain that the growth-fragmentation gives the quantum boundary length of the unexplored region in a branching SLE$_6$ exploration.

\begin{remark} In view of \cref{thm:branching}, since the branching SLE$_6$ corresponds to the scaling limit of the exploration tree of a percolated triangulation, one can expect the scaling limit of the perimeter process along the exploration tree towards the growth-fragmentation obtained in \cite{BCK18} to hold jointly with the scaling limit of the Kreweras walk towards the Brownian cone excursion as in \cite{BHS23}. Given that the latter scaling limit holds jointly with the scaling limit of the triangulation with boundary towards the Brownian disk \cite{GHS21}, we expect the three scaling limits to hold jointly. In order to combine \cite{BCK18, BHS23, DSPW25} with \cref{thm:branching} to get the joint scaling limit, the main difficulty seems to get the scaling limit of the number of spine steps along the walk towards the local time on cone-free times defined in \cite{DSPW25}. We refer the interested reader to \cite[Lemma 9.25]{BHS23} for more details on this joint convergence.
\end{remark}
\section{Notation and setup}\label{sec:notation_and_setup}
When describing maps and walks, we often follow the notation of \cite[Sections 2 \& 3]{BHS23}.

\paragraph{Triangulations.}
For integers $n\geq 0$ and $p \geq 2$, let $\triangulations{n}{p}$ be the set of all type II planar triangulations with boundary with $n$ internal vertices and perimeter $p$; that is, every $M \in \triangulations{n}{p}$ satisfies:
\begin{itemize}
    \item $M$ is a planar map without self-loops (but possibly with multiple edges). 
    \item $M$ has a distinguished $p$-sided face called the \textit{boundary face} of the map. All other faces are \textit{internal} triangles. 
    \item The edges and vertices of the boundary face form a simple cycle. These edges and vertices are called \textit{boundary} edges and vertices; all other edges and vertices are \textit{internal}.
    \item $M$ has a distinguished boundary edge called the \textit{root edge} and denoted by $\rootedge$. It is oriented so that the boundary face is on its right-hand side.    
    \item $M$ has $n$ internal vertices.
\end{itemize}
Although not strictly a triangulation, by convention we define the map consisting of only a single root edge to be the only element of $\triangulations{0}{2}$.

Let $M \in \triangulations{n}{p}$. A site percolation configuration on $M$ is a colouring of its vertices by either black or white. A configuration is said to be \textit{Dobrushin} if it satisfies the Dobrushin boundary conditions: the base vertex of the root edge is coloured white, the tip vertex of the root edge is coloured black, and the black-coloured (resp. white-coloured) boundary vertices are all consecutive. Under Dobrushin boundary conditions, when $p>2$, there is a single bicoloured boundary edge other than the root; this is called the \textit{top} edge. Under these conditions, there is path in the dual graph from  the root edge to the top edge that always has white vertices on the left hand side and black vertices on the right hand side; this is called the \textit{percolation interface}.

When $p=2$ and $n=0$, there is only a single edge, the root is also the top. The edge clockwise to the top edge is said to be to its \textit{right}, and the edge anticlockwise to the top edge is said to be to its \textit{left} (the terms originate from a planar embedding of the map where both the root and the top edges are horizontal, and the top is above the root). We define
\begin{align*}
    \percnp{n}{p} &\coloneqq \{(M,\sigma) \mid M \in \triangulations{n}{p}, \sigma \text{ is Dobrushin  on } M\}, \hspace{5pt} \\    
    \perc & \coloneqq \bigcup_{n \geq 0, ~ p \in \naturals} \percnp{n}{p} \, .
\end{align*}
Let $\overline{\perc}$ be the set of percolated triangulations in $\perc$ equipped with a marking of the boundary edges as active or inactive such that:
\begin{itemize}
    \item The top edge is active.
    \item The root edge is inactive, except when it is the same as the top edge.
    \item The active edges are consecutive.
    \item The percolation interface goes through all triangles which have an active edge.
\end{itemize}
We identify $\perc$ as a subset of $\overline{\perc}$ by saying that all the boundary edges of $(M, \sigma) \in \perc$ are inactive, except the top edge. For an example of a triangulation where there is a single white vertex on the boundary, see \cref{fig:correspondence}.

\paragraph{Peeling.}
Roughly speaking, a peeling exploration of a triangulation $M$ is a process which marks the faces of $M$ as ``explored' or ``unexplored''. Initially, only the boundary face is explored, and all internal triangles are unexplored. At each step of the process, an edge is chosen on the boundary between the unexplored and explored faces according to some (possibly random) rule, and the unexplored face adjacent to it is explored. If this face separates the unexplored faces into two disjoint sets, then one of these sets is chosen and its faces are all marked as explored as well. The process continues until all faces have been explored. 

More formally, a \textit{peeling} of a percolated triangulation $(M,\sigma)\in \perc$ (or $\overline{\perc}$, where we do not distinguish between active and inactive edges) is an ordered list of percolated planar maps
$$E_0 = (M_0,\sigma_0), E_1 = (M_1,\sigma_1), \ldots $$
where each $M_i$ has a distinguished \textit{unexplored} simple face (apart from the last step, where there is no unexplored face). The rest of the faces of $M_i$ are \textit{explored}. By convention, $M_0$ is a two-faced map made of one $p$-sided polygon, whose exterior face is the boundary face of $M$ and is explored, and whose internal face is unexplored.

The Dobrushin configuration $\sigma$ offers a natural rule for selecting which edges to explore through, called the \textit{percolation interface peeling}. Consider the unique path of edges in the dual graph passing through the internal faces of $M$ from the root to the top edge, where every dual edge in the path crosses a bicoloured edge. At each step, the peeling process chooses the next bicoloured edge in this path, called the peeled edge, and considers the triangular face $t$ of $M$ which contains the bicoloured edge and which is contained in the unexplored face. If $t$'s third vertex is internal, $M_{i+1}$ is obtained from $M_i$ by adding $t$ to $M_i$. Otherwise, $t$ separates the previously unexplored face into two parts, and $M_{i+1}$ is obtained by adding $t$ and filling the part of the map which contains no faces that intersect the percolation path. There are thus four types of steps, as illustrated in \cref{fig:peeling_exploration}:
\begin{itemize}[itemsep = 0pt]
    \item $C_a$: An internal triangle with a new white internal vertex. 
    \item $C_b$: An internal triangle with a new black internal vertex. 
    \item $L_{\ell}$: An internal triangle with a white boundary vertex $\ell$ boundary edges to the left from the peeled edge (i.e. clockwise from the peeled edge)
    \item $R_{\ell}$: An internal triangle with a black boundary vertex $\ell$ boundary edges the the right from the peeled edge (i.e. anticlockwise from the peeled edge). 
\end{itemize}
\begin{figure}[ht]
    \centering
    \includegraphics[width=0.9\linewidth]{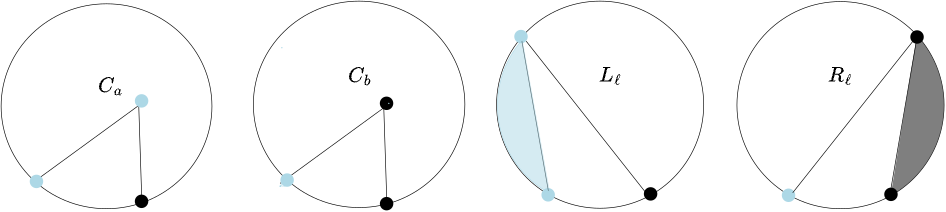}
    \caption{The four types of exploration steps $C_a, C_b, L_\ell, R_\ell$ in the peeling process. The peeled edge corresponds to the white-black edge on the bottom. In each step, a triangle is discovered and in every step of the form $L_\ell, R_\ell$, a submap with perimeter $\ell+1$ is filled in as ``explored''. In each case, the new peeled edge is the new edge with bicoloured endpoints.}
    \label{fig:peeling_exploration}
\end{figure}
The triangles discovered at steps of type $L_\ell$ or $R_\ell$ separate the unexplored face into two parts, and so these are the steps where multiple faces may be explored all at once. We denote the sequence of steps by $\peelsteps{M,\sigma}$; if $L$ steps are taken, we see this as a word $\tau = \tau_1 \cdots \tau_L$, with each $\tau_j$ in $\{C_a, C_b\} \cup\{ L_\ell, R_\ell: \ell \ge 1\}$.
The \textit{perimeter process} $P: \lb 0,L-1 \rb \to \naturals$ of a peeling $\{(M_i,\sigma_i)\}_i$ is the number of vertices (or edges, equivalently) on the boundary of the unexplored face of $(M_j,\sigma_j)$ (see \Cref{fig:peeling_perimeter}). If $M \in \percnp{n}{p}$, then $P(0) = p$ and for all $j \geq 0$,
\begin{equation}\label{eq:peeling_steps}
    P(j+1) - P(j) = \begin{cases}
        1 & \tau_{j+1} \in \{C_a, C_b\} \\
        -\ell & \tau_{j+1} \in \{L_\ell, R_\ell\} \, .
    \end{cases}        
\end{equation}

\begin{figure}[H]
    \centering
    \includegraphics[width=0.85\linewidth]{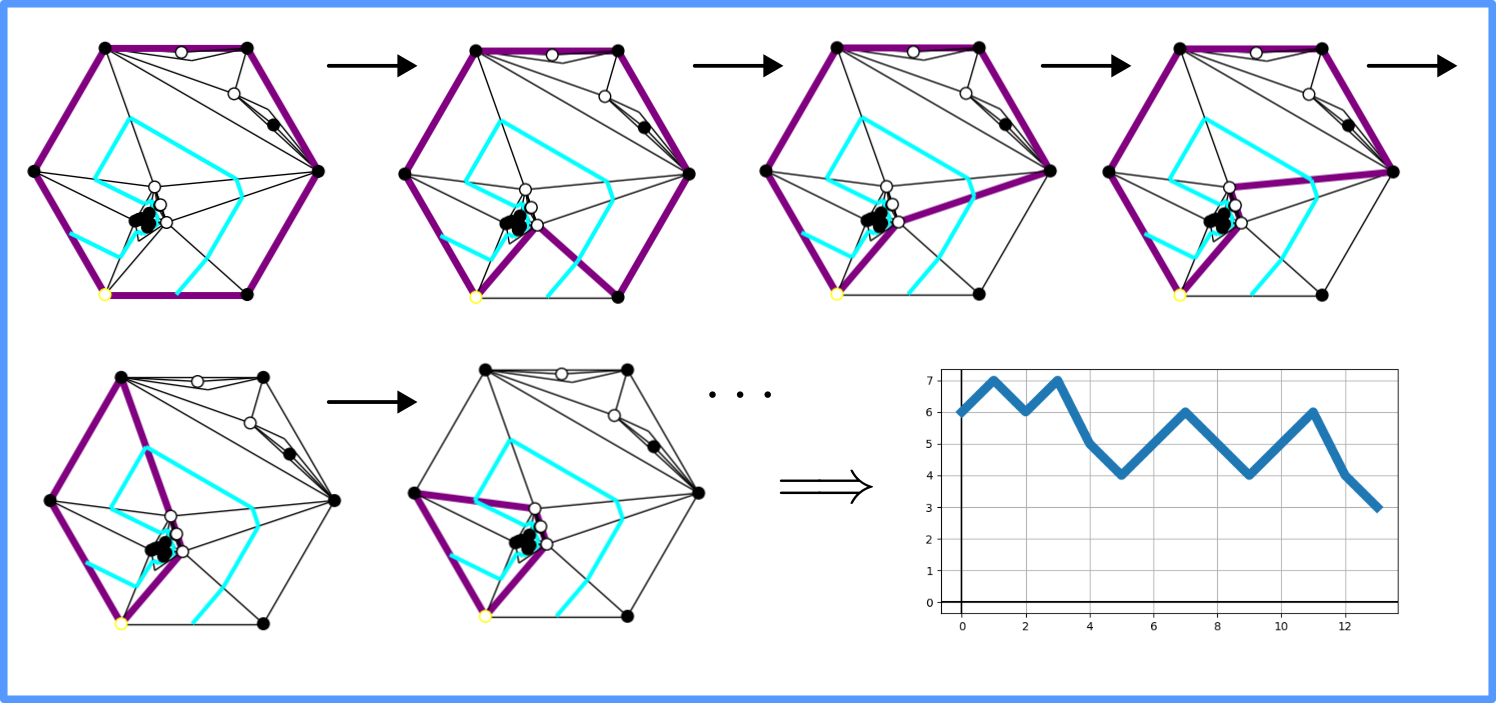}
    \caption{The perimeter of the peeling process along the percolation interface. The edges on the boundary of the unexplored face for each step of the peeling process are shown in bold purple; the first steps are $C_a, R_1, C_a, R_2, R_1\ldots$ The length of the perimeter as a function of time is shown at the bottom right.}
    \label{fig:peeling_perimeter}
\end{figure}

\paragraph{Unfilled peeling.} A variant of the percolation interface peeling is the \textit{unfilled interface peeling}. In this process, $C_a$ or $C_b$ steps are treated as before. However, if an $L_\ell$  (resp.\@ $R_\ell$) step is encountered in the $j$-th step, we glue a bicolour triangle $t$ to the peeled edge in the unexplored face, whose new vertex is white (resp.\@ black) and is placed on the boundary at a distance of $\ell$ edges to the left (resp.\@ right) from the peeled edge along the unexplored face of $E_j$. In this case, the unexplored face is split into one monochromatic white (resp.\@ black) face of degree $\ell+1$ and one bicolour face. The bicolour face is then the new unexplored face and the (one) edge it shares with $t$ is the next peeled edge. 

It is clear that the unfilled interface peeling can be recovered completely from the word $\peelsteps{M,\sigma}$ (unlike the original interface peeling, where the appearance of an $L_\ell$ or $R_\ell$ step does not tell us anything about the internal structure of the triangulation that was explored all at once).

\paragraph{Kreweras walks.} A \textit{Kreweras walk} $X:\lb 0, N \rb \to \integers^2$ of length $N \in \naturals$ starting at the origin is given by $X(0)=(0,0)$, and satisfies 
$$ X(i) - X(i-1) \in \{(1,0), (0,1), (-1,-1)\} \, $$
for all $i \in \lb 1, N\rb $. To each Kreweras walk there is an associated \textit{Kreweras word} $w = w_1\cdots w_N$ on the alphabet $\{a,b,c\}$, where the value of $w_i$ is $a$, $b$, or $c$ depending on whether $X(i)-X(i-1)$ is $(1,0)$, $(0,1)$ or $(-1,-1)$, respectively. Call $w_i$ an $a$-step (resp.\@ $b$-step, $c$-step) accordingly. We say that an $a$-step (resp.\@ $b$-step) $w_i$ and a $c$-step $w_j$ for $i<j$ are \emph{matched} if there is an equal number of $a$-steps (resp.\@ $b$-steps) and $c$-steps in $w_iw_{i+1}\cdots w_j$ and, for all $k \in \lb i, j-1 \rb$, there are strictly more $a$-steps (resp.\@ $b$-steps) than $c$-steps in $w_i w_{i+1} \cdots w_k$.

Let $\krewerasbar$ be the set of finite Kreweras walks such that in the associated Kreweras word, every $c$-step has at least one matching ($a$ or $b$) step. We will use this notation to refer to both the Kreweras walks and their corresponding words. 

For all integers $n,k \geq 0$, denote by $\krewerasnk{n}{k}$ the set of Kreweras walks in $\krewerasbar$ with $3n+2k$ steps which start at $(0,0)$, end at $(0,-k)$ and stay in the quadrant $\{(x,y) \in \integers^2: x\ge 0, y \ge -k\}$. Denote $\kreweras = \bigcup_{n,k} \krewerasnk{n}{k}$. 

\paragraph{The bijection.} Let us give a detailed description of the bijection between Kreweras walks and percolated triangulations, which was originally found by Bernardi for a specific case \cite{Ber07} and generalised in \cite{BHS23}. 

The general bijection $\overline{\Phi}: \krewerasbar \to \overline{\perc}$ is defined recursively as follows. We map the empty walk (with zero steps) to the triangulation made of only one active root edge oriented from a white vertex to a black vertex. Next, assume that we have defined the bijection $\overline{\Phi}$ for all the walks with $N-1$ steps for some $N\in \naturals$. Let $w_1\cdots w_N \in \krewerasbar$ be a walk with $N$ steps. Let $(M, \sigma)= \overline{\Phi}(w_1\cdots w_{N-1})$. Then $\overline{\Phi}(w_1 \cdots w_N)$ is defined as follows (see \cref{fig:bijection}): 
\begin{enumerate}
    \item If $w_N=a$ (resp.\@ $w_N=b$), then $\overline{\Phi}(w_1 \cdots w_N)$ is obtained by gluing a triangle on the top edge, where the new vertex of the triangle is white (resp.\@ black) and where the two new edges of the triangle are active.
    \item \label{enu:two_active_c} If $w_N=c$ and if there are active edges $e_\mathrm{left}$ and $e_\mathrm{right}$ on both the left and the right sides of the top edge, then consider the internal triangles next to $e_\mathrm{left}$ and $e_\mathrm{right}$. By definition of $\overline{\perc}$, the interface path goes through both triangles. If the triangle next to $e_\mathrm{left}$ (resp.\@ $e_\mathrm{right}$) is the last to be visited between the two, then we change the colour of the top left (resp.\@ right) vertex from white to black (resp.\@ black to white) and we glue the top edge to $e_\mathrm{right}$ (resp.\@ $e_\mathrm{left}$).
    \item \label{enu:one_active_c} If $w_N= c$ and there is no active edge on the left-hand side of the top edge, then consider the top edge, the top-right vertex $v_\mathrm{right}$ and the incident edge on the right-hand side $e_\mathrm{right}$. We first change the colour of $v_\mathrm{right}$ from black to white, making $e_\mathrm{right}$ the new top edge, then set the former top edge to be inactive. The case where there is no active edge on the right-hand side of the top edge is symmetric, with the colours reversed.
\end{enumerate}
Note that at every $a$-step or $b$-step, we create a new active edge on the boundary. Moreover, since by definition of $\krewerasbar$, every $c$-step has a matching $a$- or $b$-step, when we see a $c$-step, there is always at least one active edge other than the top edge. Therefore, the above disjunction covers all the possible cases.

The mapping $\overline{\Phi}$ is a bijection from $\krewerasbar$ to $\overline{\perc}$ \cite[Theorem 2.11]{BHS23}. Moreover, by \cite[Corollary 2.13]{BHS23}, its restriction $\Phi$ to $\kreweras$ is a bijection with $\perc$, and by \cite[Remark 2.14]{BHS23}, $\Phi$ induces a bijection between $\krewerasnk{n}{p-2}$ and percolated triangulations $\percnp{n}{p}$ with exactly one white vertex on the boundary for all $p\ge 2$ and $n\ge0$.

\begin{figure}[H]
    \centering
    \includegraphics[width=0.85\linewidth]{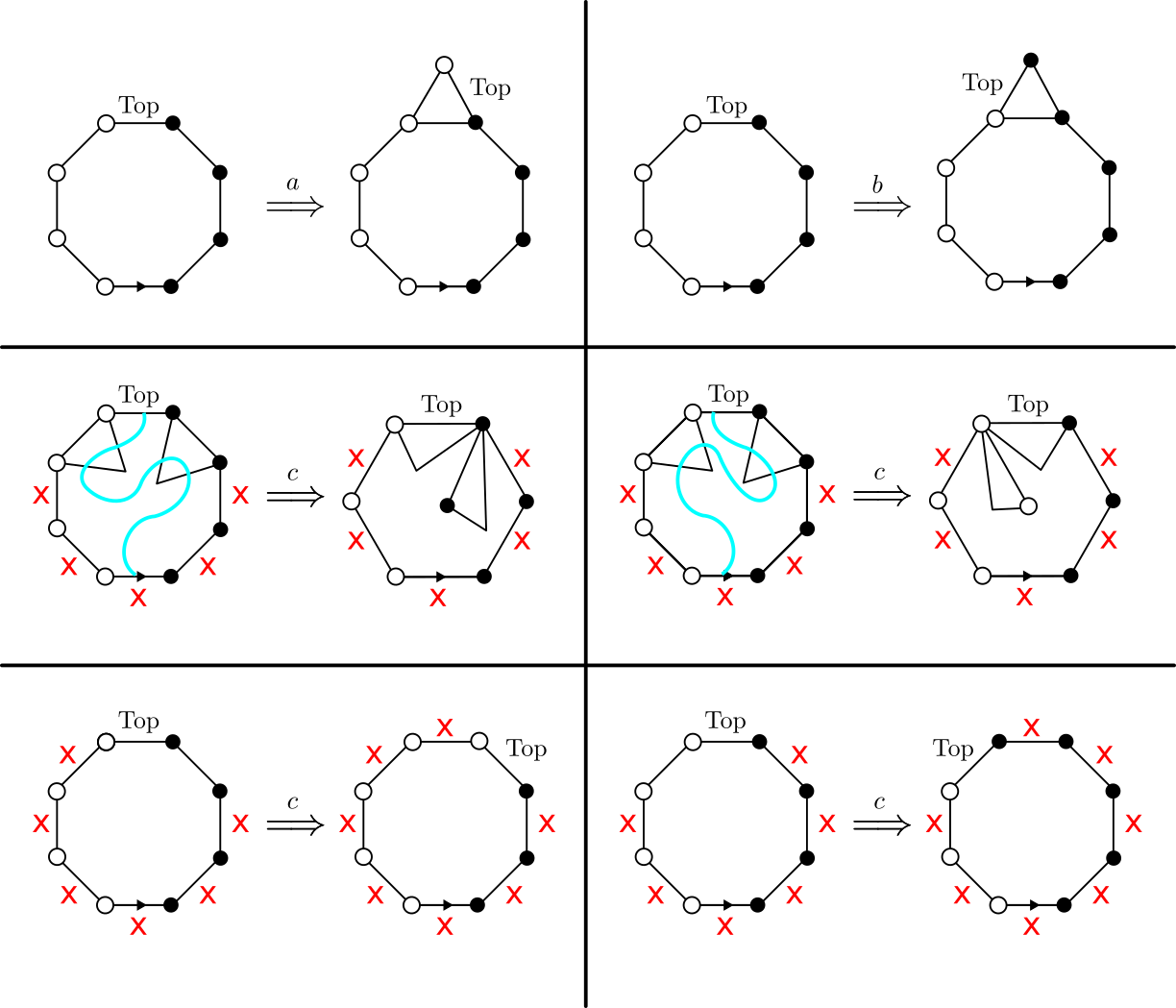}
    \caption{Building the map for the bijection, step by step. Red crosses indicate inactive edges. The top two steps depict $w_N \in \{a,b\}$, while the bottom four depict $w_N = c$; there, the construction depends on the active edges adjacent to the top, as well as the percolation interface.}
    \label{fig:bijection}
\end{figure}

\section{Proof of \texorpdfstring{\cref{thm:explicit_perimeter}}{}}
In order to analyse $\lVert  X(T(n+1))-X(N) \rVert_1$, we need to understand the walk $X$ at the spine steps. To this end, as in \cite{BHS23}, we introduce \textit{reduced Kreweras words}, defined as follows. Let $w=w_1\cdots w_N \in \krewerasbar$. If $w_i$ and $w_k$ are matched, we say that the steps $w_{i+1}\cdots w_{k-1}$ are \textit{enclosed} by the matching $w_i, w_k$. Assume that $w_k$ is a $c$-step and has two matching steps $w_i$ and $w_j$ with $i<j<k$. We say that $w_j$ is \textit{close-matched} with $w_k$ and that $(w_j, w_k)$ form a \textit{close matching}, while $w_i$ is \textit{far-matched} and $(w_i,w_k)$ form a \textit{far matching}. A step $w_i\in \{a,b\}$ is a \emph{spine step} if it is not enclosed by any close matching. Note that this definition agrees with the one given above the statement of \cref{thm:explicit_perimeter}.

For each close matching $(w_i,w_k)$ with 
 $w_i = a$, there are as many $a$-steps as $c$-steps in the subword $w_i \cdots w_k$, and no more $b$-steps than $c$-steps. Geometrically, at each step $w_{i+1},\cdots, w_{k-1}$ the walk is strictly to the right and above of the endpoint. The total displacement of the walk corresponding to $w_i\cdots w_k$ is $(0,-\ell)$, where $\ell \geq 0$ is equal to the number of $c$-steps minus the number of $b$-steps. Similarly, when $w_i =b$, the displacement of the walk is horizontal. In either case, we say that $\ell$ is the \emph{height of the cone excursion} corresponding to the matching $(w_i,w_k)$.

The \textit{reduced Kreweras word} $\widehat{\pi}(w)$ has letters $\{a,b\} \cup \{\overline{a}_\ell, \overline{b}_\ell: \ell \ge 1\}$ and is obtained from $w$ as follows: For each close matching $(w_i,w_k)$ that is not enclosed by another close matching, replace the subword $w_i \cdots w_k$  by the letter $\overline{b}_\ell$ if $w_i=a$, (resp.\@ by $\overline{a}_\ell$ if $w_i=b$), where $\ell$ is the height of the cone excursion corresponding to the matching $(w_i,w_k)$.

As the next lemma shows, the reduced Kreweras word of a map $(M,\sigma)$ is equivalent to the peeling steps $\peelsteps{M,\sigma}$.
\begin{lemma} \label[lemma]{lem:kreweras_peelstep_correspondence}
    Let $(M, \sigma) \in \perc $ and let $w = {\Phi}^{-1}(M,\sigma)$ be its corresponding Kreweras word. Then the reduced word $\widehat{\pi}(w)$ is obtained from $\peelsteps{M,\sigma}$ by replacing each instance of $C_a, C_b, L_\ell, R_\ell$ by $a,b,\overline{a}_\ell,\overline{b}_\ell$, respectively.    
\end{lemma}
The proof is given in the next subsection. With this lemma in hand, we can now prove our main theorem.

\begin{proof}[Proof of \cref{thm:explicit_perimeter} using \cref{lem:kreweras_peelstep_correspondence}]
Let $p$ be the number of sides of the outer boundary of $M$, and let $\tau = \tau_1\ldots \tau_L = \peelsteps{M,\sigma}$ be the peeling steps. As described in \eqref{eq:peeling_steps}, the $j$-th increment of the perimeter process is given by 
\begin{equation*}
    P(j+1)-P(j) = \begin{cases}
        1 & \tau_{j+1} \in \{C_a, C_b\}\\
        -\ell & \tau_{j+1} \in \{L_\ell, R_\ell\} \, ,
    \end{cases}        
\end{equation*}
with the initial condition $P(0) = p$.

Let  $Y : \lb0,L\rb \to \integers^2$ be the following walk: $Y(0) = (0,0)$, and the increments of $Y$ are obtained from the reduced Kreweras word  $\widehat{\pi} = \widehat{\pi}(w)$ by
\begin{equation}
    Y(n) - Y(n-1) = 
    \begin{cases}
        (1,0) & \widehat{\pi}_n = a \\    
        (0,1) & \widehat{\pi}_n = b \\
        (0,-\ell) & \widehat{\pi}_n = \overline{a}_\ell \\
        (-\ell,0) & \widehat{\pi}_n = \overline{b}_\ell .
    \end{cases} \label{eq:y_differences}
\end{equation}
As noted above, $X(T(n+1))=Y(n)$ for all $n \in \lb 0, L-1 \rb$. Furthermore, $X(N) = (0,-p+2)$ and so $Y(L) = (0,-p+2)$ as well. Since $X$ stays in the quadrant $\{(x,y) \in \integers^2 \mid x\geq 0, y\geq -p+2\}$, we have, for every $n \in \lb 1,L\rb$,
\begin{align*}
    \lVert Y(n) - Y(L) \rVert_1 - \lVert Y(n-1) - Y(L) \rVert_1 &= \lVert Y(n) - (0,-p+2) \rVert_1 - \lVert Y(n-1) - (0,-p+2) \rVert_1 \\
    &= Y(n)\cdot(1,1)+ p-2 - \left(Y(n-1)\cdot(1,1) +p-2 \right) \\
    &= \left(Y(n) - Y(n-1)\right) \cdot (1,1) \, .
\end{align*}
By \eqref{eq:y_differences}, this gives
\begin{equation*}
    \lVert  X(T(n+1))-X(N) \rVert_1 - \lVert  X(T(n))-X(N) \rVert_1 = \begin{cases}
        1       & \widehat{\pi}_n \in \{a, b\} \\
        -\ell   & \widehat{\pi}_n \in \{\overline{a}_\ell, \overline{b}_\ell \} \, , 
    \end{cases}
\end{equation*}
and so by \cref{lem:kreweras_peelstep_correspondence} the increments of $P$ and of $\lVert  X(T(n+1))-X(N) \rVert_1 + 2$ are identical. Since $\lVert  X(T(1))-X(N) \rVert_1 + 2 = \lVert (0,0)-(0,p+2) \rVert_1 +2 = p = P(0)$, the initial point is the same as well, and so they agree everywhere.
\end{proof}

\subsection{Proof of \texorpdfstring{\cref{lem:kreweras_peelstep_correspondence}}{}}
In order to prove the lemma, we relate both the reduced Kreweras word and the peeling steps to an intermediate object, called the spine of the map, to be defined below. \cref{lem:kreweras_peelstep_correspondence} will follow from two results: the first (\cref{lem:link_spine_peeling}) shows that the spine and the peeling process have the same triangles, while the second (\cref{lem:bhs}) shows that the steps of the reduced Kreweras walks correspond to triangles of the spine. Although the proof is not complicated, there are many objects involved; \cref{fig:proof_structure} shows the relations between them.

\begin{figure}[H]
    \centering
 \begin{tikzcd}
   w \arrow[d]\arrow[r, "\text{def}"] & \hat \pi(w) \arrow[d, "\text{\cref{lem:bhs}}"]\arrow[r, "\text{\cref{lem:kreweras_peelstep_correspondence}}"] & \arrow[l] \peelsteps{M,\sigma}\arrow[d, "\text{Peeling~steps}"]  \\
(M,\sigma) \arrow[u, "\bar \Phi"] \arrow[r,"\text{def}"]  \arrow[rr, bend right, "\text{def}"]  &  \mathrm{Spine}(M,\sigma)~~ \arrow[u] \arrow[r, "\text{\cref{lem:link_spine_peeling}}"] & \mathrm{~~UnfillPeel}(M,\sigma)\arrow[l] \arrow[u] 
\end{tikzcd}
    \caption{The relation between the objects involved in the proof of \cref{lem:kreweras_peelstep_correspondence}. Here $\mathrm{UnfillPeel}(M,\sigma)$ is the sequence of maps corresponding to the unfilled interface peeling in \cref{lem:link_spine_peeling}. \cref{lem:link_spine_peeling} shows that the lower face commutes, and \cref{lem:bhs} gives the bijection that makes the upper left square commute. \cref{lem:kreweras_peelstep_correspondence} follows upon showing that the right square commutes.} 
    \label{fig:proof_structure}
\end{figure}

Following \cite{BHS23}, the spine of a percolated triangulation $(M, \sigma) \in \overline{\perc}$, denoted by $\mathrm{Spine}(M, \sigma)$, is the percolated planar map obtained by: a) deleting the non-boundary edges and vertices which do not belong to a triangle traversed by the interface path; b) replacing every monochromatic edge incident to two triangles traversed by the interface with a double edge; c) replacing every monochromatic inactive edge incident to a triangle traversed by the interface with a double edge. The map $\mathrm{Spine}(M, \sigma)$ is then made of bicoloured triangles -- the triangles traversed by the interface -- together with some monochromatic faces. See \cref{fig:spine_of_map} for an example.
\begin{figure}[H]
    \centering
    \includegraphics[width=0.7\linewidth]{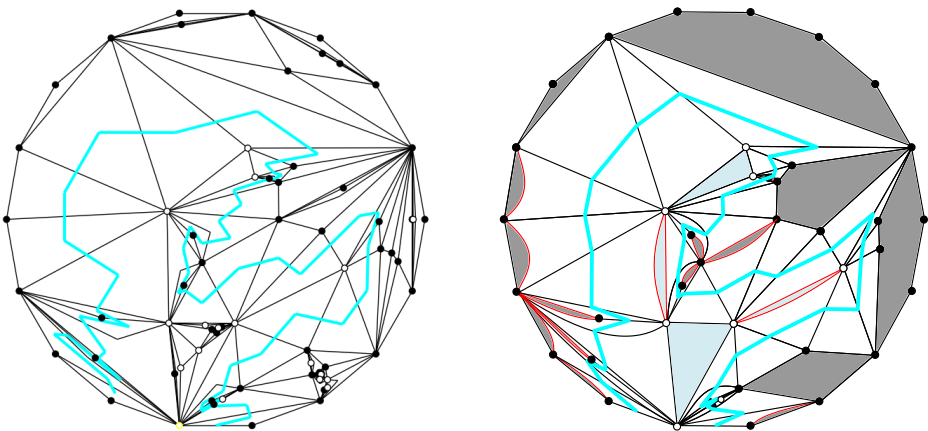}
    \caption{\textbf{Left}: a percolated triangulation $(M,\sigma)$, together with its percolation interface. \textbf{Right}: The spine $\mathrm{Spine}(M,\sigma)$. The monochromatic faces have been coloured accordingly. Double edges are shown in red; all boundary edges except the top are inactive.}
    \label{fig:spine_of_map}
\end{figure}

While the spine is obtained by deleting triangles and duplicating monochromatic edges, the unfilled peeling process $E_0, E_1, \ldots, E_L$ grows the map by adding triangles and possibly monochromatic faces, one step at a time. Upon ``peeling the spine" it turns out that these two processes are in fact equivalent.
\begin{lemma}\label[lemma]{lem:link_spine_peeling}
    Let $(M, \sigma) \in {\perc}$, and let $E_0 = (M_0, \sigma_0), \ldots, E_L = (M_L, \sigma_L)$ be the unfilled interface peeling. For all $j \in \lb 0, L\rb$, let $\tilde{E}_j$ be the map obtained by deleting from $\mathrm{Spine}(M, \sigma)$ all the bicolour triangles after the first $j$ triangles traversed by the percolation interface starting from $\rootedge$, and deleting all the monochromatic faces which have an edge which does not belong to the boundary of the map or to any of these $j$ triangles. Then 
    $$E_j = \tilde{E_j} \, .$$
    In particular, $E_L = \mathrm{Spine}(M,\sigma)$.    
\end{lemma}
\begin{proof}
Let $\tau = \tau_1 \cdots \tau _L = \peelsteps{M,\sigma}$ be the steps of the peeling process. For $j=0$, the statement holds by definition of $E_0$. Next, let us assume that the statement is true for $E_j$ for some $j \in \lb 0, L-1 \rb$. We then have two cases.
\begin{enumerate}
    \item $\tau_{j+1} \in \{C_a, C_b\}$. Then no monochromatic face is added in the unfilled interface peeling $E_{j+1}$. The same is true for $\widetilde{E}_{j+1}$, since there is no new monochromatic face whose edges lie in the boundary of the boundary face or in the edges of the $j+1$ triangles already discovered. Therefore, $\widetilde{E}_{j+1}=E_{j+1}$
    \item $\tau_{j+1} \in \{L_\ell, R_\ell\}$. Then we discover a triangle $t$ traversed by the percolation interface whose top vertex is on the boundary of the unexplored region; this creates a monochromatic face $f$ in $E_{j+1}$. Every edge of $f$ either belongs to $t$ or is one of the edges already present in $E_j$. Therefore, by the induction hypothesis, the edges of $f$ either belong to the $j+1$ triangles already discovered or to the boundary of the boundary face, and so this face is not deleted. Thus, $\widetilde{E}_{j+1}=E_{j+1}$.
\end{enumerate}
\end{proof}

The relation between the reduced Kreweras word and the spine of the map is given by the following lemma, which is essentially \cite[Theorem 3.6]{BHS23}.
\begin{lemma}[implied by the proof of Theorem 3.6 of \cite{BHS23}]\label[lemma]{lem:bhs}
    Let $(M,\sigma) \in \overline{\perc}$, and let $w = \overline{\Phi}^{-1}(M,\sigma)$. Then $\widehat{\pi}(w)$ and $\mathrm{Spine}(M, \sigma)$ are related as follows.
    \begin{itemize}
        \item The steps of $\widehat{\pi}(w)$ correspond to the bicolour triangles of $\mathrm{Spine}(M, \sigma)$ in their order of appearance along the percolation interface from $\rootedge$ to the top edge.
        
        \item Every step $\overline{a}_\ell$ (resp.\@ $\overline{b}_\ell$) of $\widehat{\pi}(w)$ corresponds to a pair $(t,f)$, where $t$ is a bicolour triangle and $f$ is a monochromatic white (resp.\@ black) face of $\mathrm{Spine}(M, \sigma)$ which shares an edge with $t$. Conversely, every monochromatic face is associated with such a step $\overline{a}_\ell$ or $\overline{b}_\ell$. The face $f$ has degree $\ell+1$, where we recall that $\ell$ is the height of the cone excursion associated with the step $\overline{a}_\ell$ (resp.\@ $\overline{b}_\ell$), and the edges around $f$ belong either to the bicolour triangles drawn up to step $\overline{a}_\ell$ (resp.\@ $\overline{b}_\ell$) or to the boundary of $M$.
        
        Furthermore, the percolated triangulation with boundary formed by $t$ and all the vertices and edges of $(M, \sigma)$ inside and on the boundary of $f$ is equal to $\overline{\Phi}(w')$, where $w'$ is the cone excursion of $w$ corresponding to $\overline{a}_\ell$ (resp.\@ $\overline{b}_\ell$), i.e.\@ the subword of $w$ starting with the associated $a$-step (resp.\@ $b$-step) and ending with the close matching $c$-step.
    \end{itemize}
\end{lemma}
Next, let us state a lemma describing which vertex has its colour changed at a $c$-step. A consequence of the proof is that each of the four cases for $c$-steps (cf. \Cref{fig:bijection}) are determined by whether the $c$-steps is far-matched or not and whether it is close matched to an $a$-step or a $b$-step.    
\begin{lemma}\label[lemma]{lem:colour_change}
    Let $w = w_1 \cdots w_N \in \krewerasbar$ be a Kreweras word, and let $j < k \in \lb 1,N \rb$ be two indices so that $w_j \in \{a,b\}$ and is close-matched with $w_k = c$. Then in the iterative construction of the bijection $\Phi$, the $c$-step $w_k$ changes the colour of the vertex of the triangle created by the step $w_j$.    
\end{lemma}
\begin{proof}
Suppose first that $w_k$ is far-matched, i.e., there exists $i \in \lb 1,N\rb$ which is far-matched with $w_k$. By \cite[Theorem 2.11(ii)]{BHS23}, $w_i$ and $w_j$
correspond to triangles with active edges on the boundary of $(M',\sigma') = \bar{\Phi}(w_1\cdots w_{k-1})$. Since $w_j$ and $w_k$ do not get reduced in $\widehat{\pi}(w_1\cdots w_{k-1})$, by \cref{lem:bhs}, the percolation interface passes through the triangles in order, and so passes through the triangle of $w_j$ after that of $w_i$. In fact, $w_j$ is the last active triangle that the interfaces passes through, since it is the last unmatched $a$- or $b$-step in $w_1 \cdots w_{k-1}$ and the order of active edges along the boundary does not change due to $c$-steps. By \cref{enu:two_active_c} in the definition of the bijection, the vertex of the triangle created by $w_j$ changes colour (and is then merged).

Suppose now that $w_k$ is not far-matched, and assume that $w_j=a$. Then there is no unmatched $b$-step in $w_1\cdots w_{k-1}$. By \cite[Theorem 2.11(iii)]{BHS23}, there is therefore no active edge on the right-hand side of the top edge in $(M', \sigma')$. Moreover, since $w_j$ is the last unmatched step of $w_1\cdots w_{k-1}$, the edge on the left-hand side of the top edge in $(M', \sigma')$ belongs to the triangle created at step $w_j$. By \cref{enu:one_active_c} in the definition of the bijection, the colour of the vertex must be changed (and the previous top edge is made inactive). We reason similarly in the symmetric case $w_j=b$.
\end{proof}

We are now in position to prove \cref{lem:kreweras_peelstep_correspondence}.
\begin{proof}[Proof of \cref{lem:kreweras_peelstep_correspondence}]
    By \cref{lem:link_spine_peeling}, the triangles along the percolation interface correspond to the peeling steps, while by \cref{lem:bhs}, the triangles along the interface correspond to the symbols of $\widehat{\pi}(w)$. Moreover, in both cases, the ordering corresponds to the order of appearance in the peeling exploration along the percolation interface. More precisely, let $\tau = \tau_1\cdots\tau_L = \peelsteps{M,\sigma}$ and let $j \in \lb 1, L \rb$.
    \begin{enumerate}
        \item If $\tau_j = C_a$, then we discover a triangle whose top vertex is an internal white vertex and there is no new monochromatic face in $E_j$ which was not present in $E_{j-1}$; by \cref{lem:bhs}, the lack of such a face implies that $\widehat{\pi}_j(w)$ cannot be $\overline{a}_\ell$ or $\overline{b}_\ell$, and so it is either $a$ or $b$. Yet it cannot be $b$: in the construction of the bijection $\Phi$, when a $b$-step is encountered, a new triangle with a \textit{black} vertex is glued to the top edge. Since the triangle discovered has a white internal vertex, a $b$-step triangle would need to change the colour of its vertex later by a matching $c$-step. But by \cref{lem:colour_change}, $c$-steps only change the colour of the vertex introduced by their close-matched $a$- or $b$-step, whereas $\widehat{\pi}_j(w) \in \{a,b\}$ implies that the $a$- or $b$-steps were not close-matched. Thus, the vertex's colour does not change, and $\widehat{\pi}_{j}(w) = a$.
       The same reasoning shows that $C_b$ corresponds to $b$. 
        \item If $\tau_j = L_\ell$ for some $\ell \ge 1$, then we know by \cref{lem:link_spine_peeling} 
        that the triangle $t$ which is the $j$-th one to be discovered by the interface, and is discovered at an $L_\ell$ step, is incident to a monochromatic white face $f$ in $\mathrm{Spine}(M, \sigma)$. The edges of $f$ belong either to the first $j$ triangles traversed by the interface or to the boundary of $M$. By \cref{lem:bhs}, we deduce that the pair $(t,f)$ is associated  with a step $\overline{a}_\ell$. Thus, the $j$-th letter of $\widehat{\pi}$ is $\overline{a}_\ell$. The same reasoning shows that $R_\ell$ corresponds to $\overline{b}_\ell$.
    \end{enumerate}
\end{proof}

\section{Branching exploration and Kreweras excursions}\label{sec:branching}
In this section, we give a formal statement of \cref{thm:branching}, as well as present its proof. 

We first formally define the branching peeling exploration. In this process, when a triangle separates the unexplored region into two, rather than filling in one part completely (as in the percolation interface peeling) or replacing the part by a single face (as in the unfilled interface peeling), the process branches, with one branch exploring the part through which the percolation interface goes, and the other branch the other part. The process is no longer linearly ordered, but rather is indexed by a tree. This tree, which lies on the dual graph, is sometimes called the exploration tree and is drawn on the right-hand side of \cref{fig:correspondence_branching}. The details are as follows.

Let $t$ be an inner triangle of $(M, \sigma) \in \perc$. Here, for our purpose, it is more convenient to record the unexplored region at every step instead of the explored region. We define a sequence of maps $\mathtt{EB}_t(n)$, where each map has a distinguished ``explored'' face. We set $\mathtt{EB}_t(0) = (M, \sigma)$, with the external face marked as explored, and start peeling the root edge $\rootedge$.

Now, suppose the map at time $k\ge 0$ is $\mathtt{EB}_t(k)$.
If the next peeling step is of the type $C_a$ (or $C_b$), then continue the exploration along the current interface. Construct 
$\mathtt{EB}_t(k+1)$ from $\mathtt{EB}_t(k)$ by removing the discovered triangle next to the root edge of $\mathtt{EB}_t(k)$, the root edge of $\mathtt{EB}_t(k+1)$ being the other bicolour edge of the discovered triangle.

If the next step is of type $R_\ell$ or $L_\ell$, recall that discovered triangle splits the unexplored region into two unexplored parts $U_1$ and $U_2$, the first one containing the rest of the percolation interface and the second one being a submap $U_2$ whose boundary vertices have the same colour. We consider the two cases:
\begin{itemize}
    \item If $t$ is in $U_1$ we continue the exploration along the interface. Set $\mathtt{EB}_t(k+1) = U_1$, with the root edge being the edge belonging to $t$.
    \item If $t$ is in $U_2$, the next peeled edge will be the monochromatic edge of the triangle discovered in the previous step. Let $\mathtt{EB}_t(k+1)$ be  the modified map obtained from $U_2$ where this edge is recoloured (from a left white vertex to a right black vertex). The root edge of $\mathtt{EB}_t(k+1)$ is next peeled edge. 
\end{itemize}
The exploration branch stops when it discovers the triangle $t$.

We now define the quantities $P_t(n)$, $T_t(n)$ and $D_t(n)$ that were informally introduced in \cref{thm:branching}.

Let $(M, \sigma) \in \perc$ be a percolated triangulation with boundary, let $t$ be an inner triangle, and let $L_t$ be the total number of triangles discovered by the exploration branch. For $n \in \lb 0, L_t -1\rb$, let $P_t(n)$ be the perimeter of the unexplored region $\mathtt{EB}_t(n)$ in the exploration branch targeting $t$. 

Let $X:\lb 0, N \rb \to \integers^2$ be the Kreweras walk corresponding to $(M,\sigma)$, and let $m_t$ be the time in $\lb 0, N-1\rb$ such that the triangle $t$ is constructed by the step $X(m_t+1)-X(m_t) \in \{(1,0), (0,1) \}$ by the Bernardi bijection. Let $L'_t$ be the total number of spine steps in $X\vert_{\lb 0, m_t \rb}$. For all $n \in \lb 1, L'_t \rb$, let $T_t(n)$ be the $n$-th spine step time of $X\vert_{\lb 0, m_t \rb}$. 

For all $n \in \lb 1, L'_t \rb$, let $D_t(n)$ be the first time $r\ge m_t$ such that $X(\lb T_t(n), r \rb) \subset X(r+1)+ \naturals^2$ if such a time exists, and $D_t(n)=N$ otherwise. Note that when such a time exists, $D_t(n)$ is the index of the first unmatched $c$ in the subword $w_{T_t(n)+1} \cdots w_N$. The example in \cref{fig:kwalksexample} shows that $T(n)$ is not always equal to $T_t(n)$.

\begin{figure}
    \centering
    \begin{subfigure}[t]{0.48\linewidth}
    \vskip 0pt
        \centering
     \includegraphics[width=0.5\linewidth]{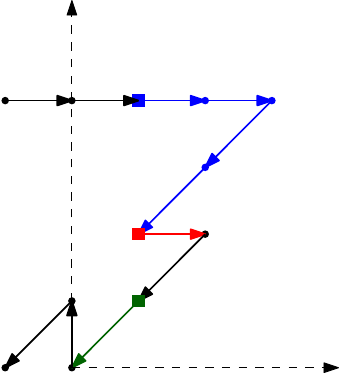}
        \caption{The red square corresponds to $m_t=7$, the blue square is the starting point of the third spine step, and the green square corresponds to $D_t(3)$. The blue trajectory is contained in the quadrant with the dashed axis.}
        \label{fig:random_walk_web_a}
    \end{subfigure}
    \hfill
    \begin{subfigure}[t]{0.48\linewidth}
    \vskip 2.18cm
    \centering
    \begin{tabular}{c|c|c|c}
        $n$ & $T_t(n)$ & $D_t(n)$ & $T(n)$\\
        \hline
       1  & 0 & 11 & 0\\
           \hline
       2  & 1 & 11 & 1
        \\   \hline
       3  & 2 & 8 & 9
         \\   \hline
       4  & 6 & 8 & -
    \end{tabular}
    \caption{If $t$ is the triangle such that $m_t=7$, then the reduced word corresponding to $X\vert_{\lb 0, m_t \rb}$ is $aa \bar a_2 a$.  }
    \label{tab:Dtn_illustration}
    \end{subfigure}
    \caption{The Kreweras walk and the values of $T_t$, $D_t$ and $T$ for the 11-letter word $aaaaccaccbc$. The reduced word is $a \bar a_4 \bar b_1$. The triangle $t$ which gives $m_t = 7$ is not on the interface path, so the reduced word of $X\vert_{\lb 0, m_t \rb}$ is not a prefix of the reduced word of the entire walk, and their spine steps are different.}
    \label{fig:kwalksexample}
\end{figure}

\begin{theorem}\label{thm:branchingdetailed}
 Let $(M, \sigma) $ be a percolated planar triangulation where the boundary has exactly one white vertex, and let $X$ be its corresponding Kreweras walk. Let $t$ be a triangle of $M$. Then $L_t=L'_t$, and for all $n \in \lb 0, L_t-1 \rb$,
    \begin{equation} \label{eq:branching_kreweras_norm}
        P_t(n)=\lVert X(T_t(n+1))- X(D_t(n+1)) \rVert_1 +2.    
    \end{equation}    
\end{theorem}

The theorem is illustrated in \cref{fig:correspondence_branching}.
\begin{figure}[H]
    \centering
    \includegraphics[width=0.75\linewidth]{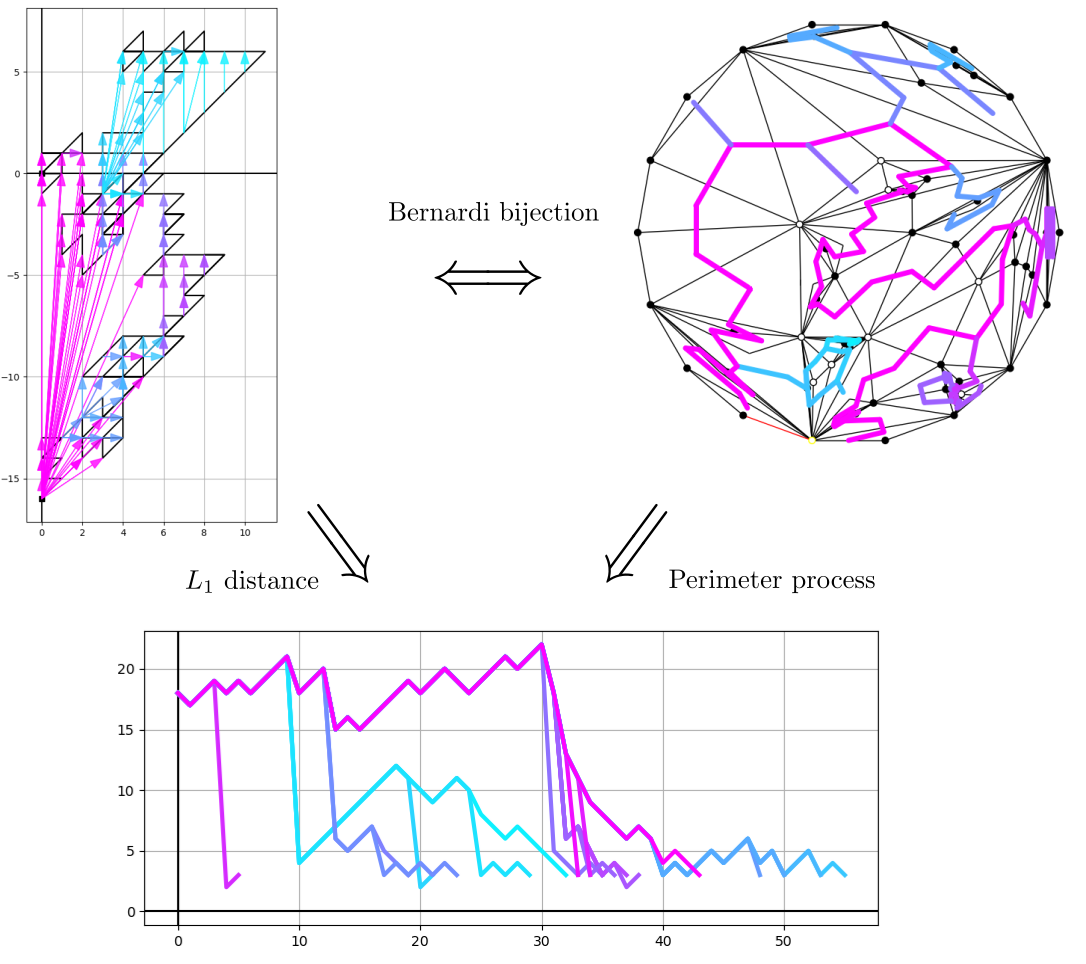}
    \caption{\textbf{Top left}: a Kreweras walk with $200$ steps. \textbf{Top right}: the equivalent percolated triangulation, along with its exploration tree. When a branch breaks off from its parent, it is drawn in a different colour. Each branch ends at a different triangle $t$. \textbf{Bottom}: the length of the perimeter of the exploration process for the maximal branches. The colours correspond to those in the triangulation. \cref{thm:branchingdetailed} states that the perimeter process of the branch reaching a triangle $t$ is given by the $L_1$ distance of the points $X(T_t(\cdot))$ from $X(D_t(\cdot))$. The difference vectors are marked in colour in the Kreweras walk. }
    \label{fig:correspondence_branching}
\end{figure}

\begin{proof}
Consider first the case where the triangle $t$ is traversed by the percolation interface.

By \cref{lem:bhs}, $t$ is constructed during a spine step of $X$. Then, $X$ cannot have a close-matched $a$- or $b$-step appear before time $m_t$ and whose matching $c$ step appears after $m_t$. In particular, the spine steps of $X\vert_{\lb 0, m_t \rb}$ are the same as the spine steps of $X$ up to time $m_t$. This implies that $L'_t=L_t \le L$, and $T_t(n)= T(n)$ for all $n \in \lb 1, L_t \rb$.

Let $n \in \lb 1, L_t \rb$, and suppose that $r \in \lb m_t, N-1 \rb$ is the first time where $X(\lb T_t(n), r \rb ) \subset X(r+1) + \N^2$. It follows that $X(r+1)-X(r)$ is a $c$-step, and this $c$-step has to be close-matched by an $a$- or $b$-step appearing before time $T_t(n)$. This is not possible since the triangle $t$ is constructed during a spine step of $X$. Therefore, no such $r$ exists, and $D_t(n)=N$. Thus, \eqref{eq:branching_kreweras_norm} is of the form $P_t(n) = \lVert  X(T(n+1))-X(N) \rVert_1 + 2$, which is the content of \cref{thm:explicit_perimeter}.

Next, assume that $t$ is not traversed by the percolation interface. Let $j_0\ge1$ be the first step in the branching peeling exploration of the form $R_\ell$ or $L_\ell$ such that $t$ lies in the part $U_2$ that is separated from the part $U_1$ containing the rest of the interface. Note that for all $j \in \lb 1, j_0 \rb$, for the same reason as in the first part of the proof, we have $T_t(j)= T(j)$ and $D_t(j)=N$. 

Let $w=w_1 \cdots w_N$ be the Kreweras word of $X$.
Since the peeling step was either $R_\ell$ or $L_\ell$, by \cref{lem:kreweras_peelstep_correspondence} the letter $w_{T(j_0)+1}\in \{a,b\}$ is closely matched with some $c$-step $w_q$. The triangle $t$ is then discovered in an $a$- or $b$-step that is enclosed by the close matching $(w_{T(j_0)+1}, w_q)$. Moreover, by the last point of \cref{lem:bhs}, the map formed by $U_2$ and the triangle discovered at step $w_{T(j_0)+1}$ is equal to $\Phi(w_{T(j_0)+1} \cdots w_q)$. 
Let $\tilde w = w_{T(j_0)+2} \cdots w_{q-1}$.
Then $\bar \Phi(\tilde w) = \mathtt{EB}_t(j_0+1)$. Indeed, by definition, $\mathtt{EB}_t(j_0+1)$ was constructed by recolouring one vertex of a monochromatic face; and by \cref{lem:colour_change}, this colour change is exactly how the matched $c$-step acts in the iterative construction of $\bar{\Phi}$.

If $t$ is traversed by the percolation interface of $\mathtt{EB}_t(j_0+1)$, then we conclude by applying \cref{thm:explicit_perimeter} as in the first case, noting that $D_t(j)=q-1$ for all $j \in \lb j_0+1, L_t \rb$. Otherwise, we recursively iterate the same reasoning as in the previous paragraph.

\end{proof}

\section{Acknowledgments}
RG acknowledges funding by the Royal Society. EK acknowledges the support of a Research Fellowship from Emmanuel College, Cambridge. FRK was supported by the Carlsberg Foundation, grant CF24-0466.
This paper was produced without the use of artificial intelligence.

\bibliographystyle{alpha}
\bibliography{biblio}

\end{document}